\documentclass{amsart}

\calclayout

\usepackage{amsmath}
\usepackage{amssymb}
\usepackage[all]{xy}
\usepackage{enumerate}
\usepackage{color}

\def\Z{{\mathbb Z}}
\def\Q{{\mathbb Q}}
\def\C{{\mathbb C}}
\def\P{{\mathbb P}}

\def\L{{\mathbb L}}

\def\cC{{\mathcal C}}

\def\cG{{\mathcal G}}

\def\M{{\mathcal M}}

\def\cP{{\mathcal P}}

\def\U{{\mathcal U}}

\def\G{\Gamma}

\def\d{{\mathfrak d}}
\def\g{{\mathfrak g}}

\def\n{{\mathfrak n}}
\def\p{{\mathfrak p}}
\def\r{{\mathfrak r}}

\def\u{{\mathfrak u}}

\def\thetadual{\check{\theta}}

\def\Ql{{\Q_\ell}}
\def\Zl{{\Z_\ell}}

\def\Gm{{\mathbb{G}_m}}
\def\Sp{{\mathrm{Sp}}}

\def\GL{{\mathrm{GL}}}
\def\GSp{{\mathrm{GSp}}}

\def\prol{{(\ell)}}
\def\un{\mathrm{un}}

\def\geom{\mathrm{geom}}

\def\top{\mathrm{top}}
\def\orb{\mathrm{orb}}

\def\ab{\mathrm{ab}}

\def\et{\mathrm{\acute{e}t}}

\newcommand\im{\operatorname{im}}
\newcommand\id{\operatorname{id}}

\newcommand\Hom{\operatorname{Hom}}

\newcommand\Spec{\operatorname{Spec}}

\newcommand\Aut{\operatorname{Aut}}

\newcommand\Gr{\operatorname{Gr}}

\newcommand\Gal{\operatorname{Gal}}

\newtheorem{theorem}{Theorem}[section]

\newtheorem{proposition}[theorem]{Proposition}
\newtheorem{corollary}[theorem]{Corollary}
\newtheorem{bigtheorem}{Theorem}
\newtheorem{bigcorollary}[bigtheorem]{Corollary}

\theoremstyle{definition}

\theoremstyle{remark}
\newtheorem{remark}[theorem]{Remark}

\begin{document}
 	
\title{Remarks on Rational Points of Universal Curves}

\author{Tatsunari Watanabe}
\address{Mathematics Department, Embry-Riddle Aeronautical University, Prescott, AZ 86301}
\email{watanabt@erau.edu}

\maketitle
\begin{abstract}
In this notes, we will give some remarks on the results in \cite{hain2} by Hain. In particular, we consider the universal curves
$\M_{g,n+1}\to \M_{g,n}$ and the sections of their algebraic fundamental groups. 

\end{abstract}

\tableofcontents
\section{Introduction}
 Denote by $\M_{g,n/\Z}$ the moduli stack of proper smooth $n$-pointed curves of genus $g$. We always assume that $2g-2+n>0$ in this paper. We have the universal curve $\pi:\cC_{g,n/\Z}\to \M_{g,n/\Z}$. For a field $k$, we denote the base change to $k$ of $\pi$ by $\cC_{g,n/k}\to\M_{g,n/k}$,  which we also denote by $\pi$.  Let $K$ be the function field of $\M_{g,n/k}$ \textcolor{black}{for $g\geq 3$}. By a generic section of $\pi$, we mean a $K$-rational point of $\cC_{g,n/k}$.  When $\mathrm{char}(k)=0$, it follows from a result of Earle and Kra \cite{EaKr} that the set of the generic sections of $\pi$ consists of exactly the tautological ones. In \cite{hain2}, Hain gave an algebraic proof to this fact, using the relative and weighted completions of the mapping class groups. \\
 \indent Let $k$ be a field and $\bar k$ be the separable closure of $k$ in a fixed algebraically closed field containing $k$. \textcolor{black}{For a geometrically connected variety or more generally a DM (Deligne-Mumford) stack $X$ over $k$, denote the \'etale fundamental group of $X$ with a geometric point $\bar x$ by $\pi_1(X, \bar x)$. }
 For a proper smooth  \textcolor{black}{geometrically connected} curve $C$ of genus $g$ over $k$ with a base point $\bar y$, there is a short exact sequence of profinite groups
 \begin{equation}1\to \pi_1(\overline{C}, \bar y)\to \pi_1(C, \bar y)\to G_k\to 1,\end{equation}
 where $\overline{C}=C\otimes_k\bar k$  and $G_k$ is the Galois group $\Gal(\bar k/k)$. Each $k$-rational point  $x$ of $C$ induces a section $x_\ast$ of $\pi_1(C, \bar y)\to G_k$ that is well-defined up to the conjugation action of an element in $\pi_1(\overline{C}, \bar y)$. The section conjecture by Grothendieck states that if $k$ is finitely generated over $\Q$ and $g\geq 2$, the function associating to each $k$-rational point $x$ the $\pi_1(\overline{C}, \bar y)$-conjugacy class $[x_\ast]$ of the section $x_\ast$ is a bijection. \\
 \indent The generic curve of type $(g,n)$ is defined to be the pull-back of the universal curve $\pi$ to the generic point $\Spec(K)$ of $\M_{g,n/k}$. Hain showed that if $k$ is a field of characteristic zero, the image of the $\ell$-adic cyclotomic character $\chi_\ell:G_k\to \Zl^\times$ is infinite for some prime number $\ell$, and $g\geq 5$, then the section conjecture holds for the generic curve of type $(g,0)$. Several examples of curves for which the section conjecture holds have been found. However, all of the curves share the same property that their associated homotopy exact sequences (1) do not split and hence they do not admit rational points. Therefore, it is an interesting question to find an example of a curve admitting a $k$-rational point for which the section conjecture holds. \\ 
 \indent Our main result \textcolor{black}{concerns} a relation between the geometric sections of $\pi$ and  the group theoretical sections of $\pi_1(\cC_{g,n/k}, \bar x)\to \pi_1(\M_{g,n/k}, \bar y)$. More precisely, we consider the universal $n$-punctured curve $\pi^o:\M_{g,n+1/k}\to \M_{g,n/k}$. There is an open immersion $j:\M_{g,n+1/k}\to \cC_{g,n/k}$ such that the complement of the image is the union of the images of the tautological sections and there is a commutative diagram:
 $$
 \xymatrix@C=1pc @R=.7pc{
 \M_{g,n+1/k}\ar@{^{(}->}[r]^-j\ar[dr]_{\pi^o}&\cC_{g,n/k}\ar[d]^{\pi}\\
 & \M_{g,n/k}.
 }
 $$ 
Let $k$ be a field of characteristic zero. Denote the tautological sections of $\pi$ by $s_1, \ldots,s_n$. Let $\bar y$ be a geometric point of $\M_{g,n/k}$. Let $C_{\bar y}$ be the fiber of $\pi$ at $\bar y$. Then the fiber $C^o_{\bar y}$ of $\pi^o$ is given by $C_{\bar y}-\{s_1(\bar y), \ldots, s_n(\bar y)\}$. Fix a geometric point $\bar x$ in $C^o_{\bar y}$.
% For a group $G$ and a prime number $\ell$, denote the pro-$\ell$ completion of $G$ by $G^\prol$.  Let $N$ be the kernel of the surjection $\pi^o_\ast: \pi_1(\M_{g,n+1/k}, \bar x)\to \pi_1(\M_{g,n/k}, \bar y)$. The kernel of $N\to N^\prol$ is characteristic in N and hence it is normal in $\pi_1(\M_{g,n+1/k}, \bar x)$. Denote the quotient $\pi_1(\M_{g,n+1/k}, \bar x)/\ker(N\to N^\prol)$ by $\pi_1(\M_{g,n+1/k}, \bar x)'$. Associated to the universal $n$-punctured curve $\pi^o$, by \cite[XIII 4.3\&4.4]{sga1}, there is an exact sequence
%$$\pi_1(C^o_{\bar y}, \bar x)^\prol\to \pi_1(\M_{g,n+1/k}, \bar x)'\to\pi_1(\M_{g,n/k}, \bar y)\to 1.$$  
Associated to the universal $n$-punctured curve $\pi^o$, there is a homotopy exact sequence:
$$1\to \pi_1(C^o_{\bar y}, \bar x)\to \pi_1(\M_{g,n+1/\textcolor{black}{\bar k}}, \bar x)\to\pi_1(\M_{g,n/\textcolor{black}{\bar k}}, \bar y)\to 1.$$
\textcolor{black}{This exact sequence induces the exact sequence of proalgebraic groups over $\Ql$ (which is the inverse limit of algebraic groups over $\Ql$)
$$1\to \cP'\to \cG^\geom_{g,n+1}\to \cG^\geom_{g,n}\to 1,$$
where $\cP'$ is the $\ell$-adic  unipotent completion of $ \pi_1(C^o_{\bar y}, \bar x)$ and $\cG^\geom_{g,n}$ is  the relative completion of $\pi_1(\M_{g,n/\bar k}, \bar \ast)$ with respect to the monodromy representation $\pi_1(\M_{g,n/\bar k}, \bar \ast)\to \Sp(H^1_\et(C_{\bar \ast}, \Zl(1)))$ with base point $\bar\ast$ (defined in \S \ref{relative}). 
}
\begin{bigtheorem} 

 Suppose that $k$ is a field of characteristic zero.
% \textcolor{red}{remove?(and  that for some prime number $\ell$, the $\ell$-adic cyclotomic character $\chi_\ell: G_k\to \Zl^\times$ has infinite image)}. 
 If $g\geq 4$ and $n \geq 0$, then \textcolor{black}{the exact sequence of proalgebraic groups
 $$1\to \cP'\to \cG^\geom_{g,n+1}\to \cG^\geom_{g,n}\to 1$$
 does not split. Consequently, the exact sequence of profinite groups
$$1\to\pi_1(C^o_{\bar y},\bar x)\to\pi_1(\M_{g,n+1/\bar k}, \bar x)\to \pi_1(\M_{g,n/\bar k}, \bar y)\to 1$$
does not split. }

\end{bigtheorem}
Let $S_g$ be a compact oriented topological surface of genus $g$ and $P$ a subset of $S_g$ consisting of $n$ distinct points. 
Denote the mapping class group of $S_g$ fixing $P$ pointwise by $\G_{g,n}$ and the topological fundamental group of the $S_g-P$ by $\Pi'^\top$.  

\begin{bigcorollary}
	If $g\geq 4$ and $n\geq 0$, then the exact sequence 
	$$1\to \Pi'^\top\to\G_{g,n+1}\to \G_{g,n}\to 1$$
	does not split.
\end{bigcorollary}
\textcolor{black}{The case  $g\geq 2$ and $n=0$ of the corollary  was} proved in \cite[Cor.~5.11]{FaMa}. For $n >0$, the result should be well-known to experts, but the author is not aware of good references. Theorem 1 is essentially a corollary of the result \cite[Thm.~12.6]{hain0} on the Lie algebra structure of $\Gr^L_\bullet\p_{g,n}$. 
 Hain used the information to compute all the graded Lie algebra sections of a surjection between \textcolor{black}{certain} two-step nilpotent Lie algebras. \textcolor{black}{These  Lie algebras} are constructed from the universal curve $\pi$ \textcolor{black}{using} the relative and weighted completions applied to \textcolor{black}{the} fundamental groups of \textcolor{black}{$\cC_{g,n}$ and $\M_{g,n}$}. \textcolor{black}{A key result for the proof of Theorem 1 is Proposition \ref{no sections of dgn lie algebra projections}, where we prove the analogue of Hain's computation for the universal $n$-punctured curve $\pi^o$.} Theorem 1 can be considered as \textcolor{black}{a variation} of Hain's result \cite[Thm.~2]{hain2} \textcolor{black}{in terms of} the universal $n$-punctured curve $\pi^o$. 

%%%%%%%%%%%%%%%%%%%%%%%%%%%%%%%%%%%%%%%%%%%%%%%%%%%%%%%%%%%%%%%%%%%%%%%%%%%%%%%%%%%%%%%%%%%%%%%%%%%%%%%%%%%%%%%%%%%%%%%
%%%%%%%%%%%%%%%%%%%%%%%%%%%%%%%%%%%%%%%%%%%%%%%%%%%%%%%%%%%%%%%%%%%%%%%%%%%%%%%%%%%%%%%%%%%%%%%%%%%%%%%%%%%%%%%%%%%%%%%
\section{Families of Curves and Monodromy Representation}
\subsection{Mapping class groups}Assume that $2g-2+n > 0$.  Let $S_g$ be a compact oriented topological surface of genus $g$. Let $P$ be a subset of $S_g$ consisting of $n$ distinct points. The mapping class group denoted by $\G_{g,P}$ is the group of isotopy classes of orientation-preserving diffeomorphisms of $S_g$ fixing $P$ pointwise. By classification of surfaces, the group $\G_{g,P}$ is independent of the subset $P$, and hence it will be denoted by $\G_{g,n}$ in this paper. There is a natural action  of  $\G_{g,n}$ on $H_1(S_g,\Z)$ preserving the intersection pairing $\langle~,~\rangle$ and thus there is a natural representation 
$$\G_{g,n}\to \Aut(H_1(S_g,\Z), \langle~,~\rangle)=\Sp(H_1(S_g,\Z)).$$
It is well-known that this homomorphism is surjective. 

\subsection{The universal curve of type $(g,n)$}
 Suppose that $g$ and $n$ are nonnegative integers satisfying $2g-2+n >0$. Suppose that $S$ is a scheme. A curve $f:C\to S$ of \textcolor{black}{type $(g, n)$} is a proper smooth morphism, whose geometric fiber is a connected one-dimensional scheme of arithmetic genus $g$, \textcolor{black}{with $n$ disjoint sections $s_1,\ldots,s_n: S\to C$.}   
 The complement in $C$ of the images of the sections $s_i$ is denoted by $f^o:C^o\to S$. The moduli stack of curves of type $(g,n)$ is denoted by $\M_{g,n}$ and the universal curve of type $(g,n)$ over $\M_{g,n}$ is denoted by $\pi:\cC_{g,n}\to\M_{g,n}$. It is proved in \cite{knu} that the stack $\M_{g,n}$ is a smooth DM-stack over $\Z$. For each curve $f:C\to S$ of type $(g,n)$, there exists a unique morphism %\textcolor{black}{(unique up to isomorphism)}
 $\phi:S\to \M_{g,n}$ such that the curve $f$ is \textcolor{black}{isomorphic} to the pull-back of the universal curve $\pi:\cC_{g,n}\to\M_{g,n}$ along the morphism $\phi$. The universal curve $\cC_{g,n}\to \M_{g,n}$ admits $n$ tautological sections and the complement of the images of the tautological sections  is defined to be the universal $n$-punctured curve  $\pi^o:\M_{g,n+1}\to \M_{g,n}$. The curve $f^o:C^o\to S$ is then isomorphic to the fiber product $S\times_{\M_{g,n}}\M_{g, n+1}\to S$.  
\subsection{Exact sequences and monodromy  associated to curves of type $(g,n)$}
Let $k$ be a field of characteristic zero. %Fix a prime number $\ell$ distinct from $\mathrm{char}(k)$. 
Suppose that $S$ is a geometrically connected scheme over $k$. Let $f:C\to S$ be a curve of type $(g, n)$. Let $\bar y$ be a geometric point of $S$ and let $C_{\bar y}$  and $C^o_{\bar y}$ be the fibers over $\bar y$ of $f$ and $f^o$, respectively. Note that $C^o_{\bar y}=C_{\bar y}-\{s_1(\bar y),\ldots,s_n(\bar y)\}$. Let $\bar x$ be a geometric point of $C^o_{\bar y}$. We also consider $\bar x$ as a geometric point of $C_{\bar y}$ via the open immersion $C^o_{\bar y}\to C_{\bar y}$.  Denote $\pi_1(C_{\bar y}, \bar x)$  and $\pi_1(C^o_{\bar y},\bar x)$ by $\Pi$ and $\Pi'$, respectively. 
%Recall that $\pi_1(C, \bar x)'=\pi_1(C, \bar x)/\ker(N\to N^\prol)$, where $N$ is the kernel of the surjection $\pi_1(C, \bar x)\to\pi_1(S, \bar y)$.  The group $\pi_1(C^o, \bar x)'$ is defined similarly. 
\textcolor{black}{Although the following result is well-known, we state it here for its importance in this paper. For example, the case of $f^o$ follows immediately from \cite[Lemma~2.1 ]{mata} and a similar argument can be applied to the case of $f$. } 
\begin{proposition}\label{exact seq} If $g>1$, then there are exact sequences of profinite groups:
$$1\to \Pi\to \pi_1(C,\bar x)\to \pi_1(S, \bar y)\to 1$$
and
$$1\to \Pi'\to \pi_1(C^o, \bar x)\to \pi_1(S, \bar y)\to 1.$$
\end{proposition}
Set $H_\Zl = H^1_\et(C_{\bar y}, \Zl(1))$. Note that as a $G_k$-module, $H_\Zl$ is isomorphic to $\pi_1(C_{\bar y}, \bar x)^\ab \otimes \Zl$ and is the dual of $H^1_\et(C_{\bar y}, \Zl)$, where the superscript $\ab$ denotes abelianization. The cohomology group $H_\Zl$ is a free $\Zl$-module of rank $2g$ endowed with the cup product pairing $\theta: \Lambda^2H_\Zl\to \Zl(1)$. Recall that the general symplectic group $\GSp(H_\Zl)$ is defined as
$$\GSp(H_\Zl) =\{T \in \GL(H_\Zl)| ~T^\ast \theta = \lambda_T\theta \text{ for some } \lambda_T \in \Zl^\times \}.$$
Associating the scalar $\lambda_T$  to $T$ defines a surjective homomorphism $\tau:\GSp(H_\Zl)\to \Zl^\times$. The symplectic group $\Sp(H_\Zl)$ is defined to be the kernel of $\tau$. 
\textcolor{black}{Now, since $R^1f_\ast\Zl(1)$ is a smooth $\ell$-adic sheaf over $S$, we have a representation}
%By Proposition \ref{exact seq},  associated to the curve $f$, there is a natural representation 
%$$\pi_1(S, \bar y)\to \Aut(H_1(\Pi)),$$
%which is induced by the conjugation action of $\pi_1(C, \bar x)$ on $\Pi$. 
%Since $H_\Zl \cong H^1(\Pi, \Zl(1)) =\Hom(\Pi, \Zl(1))$, we obtain a representation 
$$\rho_{\bar y}: \pi_1(S, \bar y)\to \GSp(H_\Zl),$$
and there is a commutative diagram
$$\xymatrix@C=1em@R=1em{
1\ar[r]&\pi_1(S\otimes_k\bar k, \bar y)\ar[r]\ar[d]^{\rho^\geom_{\bar y}}&\pi_1(S, \bar y)\ar[r]\ar[d]^{\rho_{\bar y}}&G_k\ar[r]\ar[d]^{\chi_\ell}&1\\
1\ar[r]&\Sp(H_\Zl)\ar[r]&\GSp(H_\Zl)\ar[r]^-\tau&\Zl^\times\ar[r]&1,
}
$$
where we denote the representation at left by $\rho^\geom_{\bar y}$. 
 In this paper, the weighted completion (see section \ref{weight}) of the monodromy representation $\rho_{\bar y}$ associated to the universal curve plays a key role. \textcolor{black}{The following result is also well-known and it  follows easily from the fact that the natural representation $\G_{g,n}\to \Sp(H_1(S_g,\Z))$ is surjective.}  \textcolor{black}{For example, see   \cite[\S 8]{hain2}.} Set $H_\Ql:=H_\Zl\otimes\Ql$. 
\begin{proposition}\label{monodromy density} 
  Let $k$ be a field of characteristic zero.  If $g\geq 2$,  $\ell$ is a prime number, and if the image of the $\ell$-adic cyclotomic character $\chi_\ell: G_k\to \Zl^\times$ is infinite, then the image of the monodromy representation 
 $$\rho_{\bar y}:\pi_1(\M_{g,n/k}, \bar y)\to \GSp(H_\Ql)$$ is Zariski-dense in $\GSp(H_\Ql)$. 
\end{proposition}
%\begin{proof} For a commutative ring $A$, let $\Sp(A) =\Sp(H_1(S_g, A))$. Taking profinite completions of the natural representation $\G_{g,n}\to \Sp(\Z)$, we obtain a continuous surjective homomorphism $\rho^\wedge: \G_{g,n}^\wedge\to \Sp(\Z)^{\wedge}$. By \cite[Prop.~2.2]{relative prol}, there are natural isomorphisms, 
%$\Sp(\Z)^\wedge \cong \Sp(\Z^\wedge)$ and $\Sp(\Z^\wedge)\cong \prod_{\ell \text{ prime}}\Sp(\Zl)$. Thus composing with $\rho^\wedge$, we obtain a surjective continuous homomorphism $\rho_\ell: \G_{g,n}^\wedge\to \Sp(\Zl)$. Identifying $\Sp(H_\Zl)$ with $\Sp(\Zl)$, $\rho_{\ell}$ agrees with the representation $\rho^\geom_{\bar y}: \pi_1(\M_{g,n/\bar k}, \bar y)\to \Sp(H_\Zl)$ up to the conjugation action of an element of $\pi_1(\M_{g,n/\bar k}, \bar y)$. Therefore, the image of $\rho^\geom_{\bar y}$ is Zariski dense. Since the diagram
%$$
%\xymatrix@C=1pc @R=.7pc{
%1\ar[r]&\pi_1(\M_{g,n/\bar k}, \bar y)\ar[d]^-{\rho^\geom_{\bar y}}\ar[r]&\pi_1(\M_{g,n/k}, \bar y)\ar[r]\ar[d]^{\rho_{\bar y}}&G_k\ar[r]\ar[d]^{\chi_\ell}&1\\
%1\ar[r]&\Sp(H_\Ql)\ar[r] &\GSp(H_\Ql)\ar[r]^-\tau& \Gm_{/\Ql}\ar[r]&1,
%}
%$$
%commutes and the image of $\chi_\ell$ is infinite, it follows that the image of $\rho_{\bar y}$ is Zariski dense. 
%
%
%\end{proof}
%

%%%%%%%%%%%%%%%%%%%%%%%%%%%%%%%%%%%%%%%%%%%%%%%%%%%%%%%%%%%%%%%%%%%%%%%%%%%%%%%%%%%%%%%%%%%%%%%%%%%%%%%%%%%%%%%%%%%%%%%
%%%%%%%%%%%%%%%%%%%%%%%%%%%%%%%%%%%%%%%%%%%%%%%%%%%%%%%%%%%%%%%%%%%%%%%%%%%%%%%%%%%%%%%%%%%%%%%%%%%%%%%%%%%%%%%%%%%%%%%
\section{Review of a minimal presentation of Lie Algebra $\Gr_\bullet\p_{g,n}$}

\subsection{Pure braid groups on $S_g$} Let $g$ and $n$ be  positive integers. The configuration space $F_{g,n}$ of $n$ \textcolor{black}{ordered} points on $S_g$ is defined to be 
$$F_{g,n}=S_g^n-\bigcup_{i\not=j}\left\{x_i=x_j\right\}.$$
The topological fundamental group $\pi_1(F_{g,n}, x)$ with a fixed base point $x$ of $F_{g,n}$ is  the group of pure braids with $n$ strings on $S_g$. The topological fundamental group $\pi_1^\top(F_{g,n},x)$ is denoted by $\pi^\top_{g,n}$. 

\subsection{The unipotent completion $\cP_{g,n}$ of $\pi^\top_{g,n}$ and its Lie algebra $\p_{g,n}$}
Let $F$ be a field of characteristic zero. Suppose that $\G$ is a group. The unipotent completion
%also known as Malcev completion \cite{malcev}, 
 of $\G$ is a prounipotent group $\G^\un$ over $F$ equipped with a natural Zariski-dense homomorphism $\rho: \G\to \G^\un(F)$ satisfying the universal property: for a Zariski-dense homomorphism $\rho_U: \G\to U(F)$ into the $F$-rational points of a prounipotent $F$-group $U$, there is a unique morphism $\phi_U:\G^\un\to U$ such that $\rho_U=\phi_U(F)\circ \rho.$ When $H_1(\G, F)$ is finite-dimensional, there is a concrete construction of $\G^\un$ using the completed group algebra $F[\G]^\wedge$ (see \cite[Appendix A]{malcev} and \cite{bouse}). \\
 \indent Let $\cP_{g,n}$ be the unipotent completion of $\pi^\top_{g,n}$ over $\Q$. For any field $F$ of characteristic zero, the unipotent completion $\cP_{g,n/F}$ of $\pi^\top_{g,n}$ over $F$ is canonically isomorphic to $\cP_{g,n}\otimes_\Q F$ by \textcolor{black}{the} universal property \textcolor{black}{of the unipotent completion}. Denote the Lie algebra of $\cP_{g,n}$ by $\p_{g,n}$. It is the inverse limit of finite-dimensional nilpotent Lie algebras over $\Q$, which we call a pronilpotent Lie algebra over $\Q$. A key ingredient of the proof of Theorem 1 is the structure of the graded Lie algebra $\Gr_\bullet\p_{g,n}$ associated with its lower central series $L^\bullet\p_{g,n}$.  
 \subsection{A minimal presentation}
 Firstly, we recall the following fact:
 \begin{proposition}[{\cite[Prop.~2.1.]{hain0}}] \label{abelianization pure weight -1}If $g\geq 0$ and $n\geq 0$, then there are isomorphisms
 $$H_1(\p_{g,n})\cong H_1(\pi^\top_{g,n},\Q)\cong H_1(S_g^n,\Q)\cong\bigoplus_{i=1}^nH_1(S_g, \Q)_i.$$ 
 \end{proposition}
  Secondly,  we recall a minimal presentation of a graded Lie algebra with negative weights.  \textcolor{black}{We note that such a Lie algebra is necessarily pronilpotent.}  \textcolor{black}{For a vector space $V$, denote the free Lie algebra generated by $V$ by $\L(V)$.} If $\n$ is a graded Lie algebra with negative weights, then there is a Lie algebra surjection $\phi: \L(H_1(\n))\to \n$ from the free Lie algebra generated by $H_1(\textcolor{black}{\n})$ onto $\n$.  Denote the kernel of $\phi$ by $R$. A minimal presentation of $\n$ is given by 
 $$\L(H_1(\n))/R\cong \n.$$
  For $u\in H_1(S_g, \Q)$, we will denote the corresponding vector in $H_1(S_g, \Q)_i$ by $u^{(i)}$. Fix a symplectic basis $a_1, b_1,\ldots, a_{\textcolor{black}{g}},b_{\textcolor{black}{g}}$ for $H_1(S_g)$. Denote the elements $\sum_{k=1}^{\textcolor{black}{g}}[a_k^{(i)}, b_k^{(i)}]$   and $\sum_{k=1}^{\textcolor{black}{g}}[a_k^{(i)}, b_k^{(j)}]$ for $i\not=j$ in $\L(H_1(S_g)^{\oplus n})$ by $\Theta_i$ and $\Theta_{ij}$, respectively. For $u$ and $v$ in $H_1(S_g
  %,\textcolor{red}{remove?(\Z)}
   )$, we will denote the intersection number of $u$ and $v$ by $(u,v)$. Denote the lower central series of $\p_{g,n}$ by $L_\bullet\p_{g,n}$.  We have the following minimal presentation of $\Gr^L_\bullet\p_{g,n}$ by Hain:
   \begin{theorem}[{\cite[Thm.~12.6 ]{hain0}}] \label{presentation for configuration space}For $g\geq 1$ and $n\geq 0$, there is an isomorphism of graded Lie algebras
 $$\Gr^L_\bullet \p_{g,n}\cong \L(H_1(S_g)^{\oplus n})/R,$$
 where $R$ is the Lie ideal generated by the vectors of the form
 \begin{align*}
 [u^{(i)}, v^{(j)}]-[u^{(j)}, v^{(i)}]\hspace{1in}\text{for all $i$ and $j$};& \\
 [u^{(i)}, v^{(j)}]-\frac{(u,v)}{g}\sum_{k=1}^g\Theta_{ij}\hspace{1in}\text{for }i\not=j;& \\
 \Theta_i + \frac{1}{g}\sum_{j\not=i}\Theta_{ij}\hspace{1in}\text{for }1\leq i\leq n,&
 \end{align*}

 with arbitrary $u,v\in H_1(S_g)$.

 \end{theorem}

%%%%%%%%%%%%%%%%%%%%%%%%%%%%%%%%%%%%%%%%%%%%%%%%%%%%%%%%%%%%%%%%%%%%%%%%%%%%%%%%%%%%%%%%%%%%%%%%%%%%%%%%%%%%%%%%%%%%%%%
%%%%%%%%%%%%%%%%%%%%%%%%%%%%%%%%%%%%%%%%%%%%%%%%%%%%%%%%%%%%%%%%%%%%%%%%%%%%%%%%%%%%%%%%%%%%%%%%%%%%%%%%%%%%%%%%%%%%%%%
\section{Two-step nilpotent Lie algebras associated to Universal ($n$-punctured) Curves}\label{weight}
The weighted completion of a profinite group is introduced and developed by Hain and Matsumoto  and a detailed introduction of the theory and properties are included in \cite{wei} and the reader can find a brief introduction in \cite[\S6 \& \S 7]{hain2} on which this paper is based. 
\subsection{Weighted completion applied to the fundamental groups of universal curves}\label{Applications to universal curves}
Suppose that $g\geq 2$ and $n\geq 0$. Let $k$ be a field of characteristic zero such that for some prime $\ell$, the $\ell$-adic cyclotomic character $\chi_\ell:G_k\to \Zl^\times$ has infinite image. % or $k$ be a finite field of characteristic $p$.
  Note that when $k$ is %a finite field, 
a number field or a local field $\Q_q$ with $q\not=\ell$, $\chi_\ell$ has infinite image.
Fix a geometric point $\bar y$ in $S$.  By Proposition \ref{monodromy density}, we have a Zariski-dense representation
$$\rho_{\bar y}: \pi_1(\M_{g,n/k},\bar y)\to \GSp(H_\Ql).$$
Define a central cocharacter $\omega:\Gm\to \GSp(H_\Ql)$ by sending $z$ to $z^{-1}\id$. This technical definition of $\omega$ is made so that \textcolor{black}{the weight determined by $\omega$ agrees with the weights defined by Hodge theory and Galois theory.} Denote the weighted completion of $\pi_1(\M_{g,n/k},\bar y)$ with respect to $\rho_{\bar y}$ and $\omega$ by 
$$(\cG_{g,n},\,\, \tilde{\rho}_{\bar y}:\pi_1(\M_{g,n/k},\bar y)\to \cG_{g,n}(\Ql)).$$
The completion $\cG_{g,n}$ is a negatively weighted\footnote{$H_1(\U_{g,n})$ admits only negative weights as a $\Gm$-representation via $\omega$.} extension of $\GSp(H_\Ql)$ by a prounipotent group over $\Ql$, which we denote by $\U_{g,n}$: there is a \textcolor{black}{split} exact sequence of proalgebraic groups over $\Ql$
$$1\to \U_{g,n}\to \cG_{g,n}\to \GSp(H_\Ql)\to 1,$$
\textcolor{black}{where any two splittings are conjugate by an element of $\U_{g,n}$.}
\textcolor{black}{A key property of $\cG_{g,n}$ used in this paper is that every finite-dimensional  $\cG_{g,n}$-module admits  a natural weight filtration and so do the ind- and pro-objects in the category of $\cG_{g,n}$-modules. The construction of the natural weight filtration is explained in (\cite[\S3.1]{wei}).}
\subsection{Relative completion of $\pi_1(\M_{g,n/\bar k})$ and the $\cG_{g,n}$-modules $\u^\geom_{g,n}$}\label{relative} 
%Fix a separable closure $\bar k$ of $k$, where $k$ is as above.  
%Denote $\M_{g,n/\bar k}$ by $\bar S$. 
The key object in this paper is the Lie algebra of the relative completion of $\pi_1(\M_{g,n/\bar k}, \bar y)$, which is a pro-object  in the category of $\cG_{g,n}$-modules, admitting  a natural weight filtration. A detailed review and basic properties of relative completion are given in \cite{hain5}. \\
\indent Recall that  the image of the  monodromy representation $\rho^\geom_{\bar y}: \pi_1(\M_{g,n/\bar k}, \bar y)\to \Sp(H_\Ql)$ is Zariski-dense in $\Sp(H_\Ql)$. 
%Thus, we may consider $\rho^\geom_{\bar y}$ as a homomorphism into $\Sp(H_\Ql)$. As in the case of $\rho_{\bar y}: \pi_1(S, \bar y)\to \GSp(H_\Ql)$, when $\ell$ is different from $\mathrm{char}(k)$, it follows from \textcolor{red}{\cite[Prop.??]{wat1} }that the homomorphism $\rho^\geom_{\bar y}$ has Zariski-dense image. 
The relative completion of $\pi_1(\M_{g,n/\bar k}, \bar y)$ with respect to $\rho^\geom_{\bar y}$ consists of a proalgebraic group $\cG^\geom_{g,n}$ over $\Ql$ and a homomorphism 
$$\tilde{\rho}^\geom_{\bar y}:\pi_1(\M_{g,n/\bar k}, \bar y)\to \cG^\geom_{g,n}(\Ql).$$
 The proalgebraic group $\cG^\geom_{g,n}$ is an extension of $\Sp(H_\Ql)$ by a prounipotent $\Ql$-group $\U^\geom_{g,n}$, that is, there is an exact sequence of proalgebraic groups over $\Ql$:
$$1\to \U^\geom_{g,n}\to \cG^\geom_{g,n}\to \Sp(H_\Ql)\to 1.$$
%and the homomorphism $\tilde{\rho}^\geom_{\bar y}$ satisfies a universal property similar to the one for weighted completion. 
Denote the Lie algebras of $\cG^\geom_{g,n}$, $\U^\geom_{g,n}$, and $\Sp(H_\Ql)$ by $\g^\geom_{g,n}, \u^\geom_{g,n}$, and $\r$, respectively. We list \textcolor{black}{the} results needed in this paper:

\begin{proposition}[{\cite[Prop.~8.4]{hain2}, \cite[Thm.~3.9  ]{wei}}] \label{comparison}With notations as above, 
%\begin{enumerate}
the conjugation action of $\pi_1(\M_{g,n/ k}, \bar y)$ on $\pi_1(\M_{g,n/\bar k},\bar y)$ induces an adjoint action of $\cG_{g,n}$ on $\g^\geom_{g,n}$, and hence on $\u^\geom_{g,n}$. Therefore, the Lie algebras $\g^\geom_{g,n}$ and $\u^\geom_{g,n}$ admit weight filtrations $W_\bullet$ satisfying 
\begin{enumerate}
\item $\g^\geom_{g,n}=W_{0}\g^\geom_{g,n}, \,\,\u^\geom_{g,n}=W_{-1}\u^\geom_{g,n}=W_{-1}\g^\geom_{g,n}, \,\,\text{ and }\,\,\Gr^W_0\g^\geom_{g,n}=\r;$
\item the action of $\cG_{g,n}$ on $\Gr^W_\bullet\u^\geom_{g,n}$ factors through $\cG_{g,n}\to \GSp(H_\Ql)$;
\item each $\Gr^W_m\u^\geom_{g,n}$ is of weight $m$ as a $\Gm$-representation via $\omega$.

\end{enumerate}

%\item Let $F$ be a field of characteristic zero such that the $\ell$-adic cyclotomic character $\chi_\ell:G_F\to \Zl^\times$ has infinite image for some prime $\ell$. There are weight filtration preserving isomorphisms of pronilpotent Lie algebras over $\Ql$
%$$\u^\geom_{g,n/\bar F}\cong \u^\geom_{g,n/\bar \Q}\cong\u^\geom_{g,n/\bar{\Q}_p}\cong \u^\geom_{g,n/\bar{\F}_p}$$
%that induce $\GSp(H_\Ql)$-equivariant isomorphisms of graded Lie algebras
%$$\Gr^W_\bullet\u^\geom_{g,n/\bar F}\cong\Gr^W_\bullet\u^\geom_{g,n/\bar \Q}\cong\Gr^W_\bullet\u^\geom_{g,n/\bar{\Q}_p}\cong\Gr^W_\bullet\u^\geom_{g,n/\bar{\F}_p}.$$
%Furthermore, these isomorphisms respect the projection maps: the diagram
%$$\xymatrix{
%\u^\geom_{g,n+1/\bar F}\ar[r]^{\cong}\ar[d]&\u^\geom_{g,n+1/\bar \Q}\ar[r]^{\cong}\ar[d]&\u^\geom_{g,n+1/\bar{\Q}_p}\ar[r]^{\cong}\ar[d]&\u^\geom_{g,n+1/\bar{\F}_p}\ar[d]\\
%\u^\geom_{g,n/\bar F}\ar[r]^{\cong}&\u^\geom_{g,n/\bar \Q}\ar[r]^{\cong}&\u^\geom_{g,n/\bar{\Q}_p}\ar[r]^{\cong}&\u^\geom_{g,n/\bar{\F}_p}
%}
%$$
%commutes, where the vertical maps are induced by the projection $\M_{g,n+1}\to\M_{g,n}$. 
%\end{enumerate}
\end{proposition}
 %From now on, we will simply denote the Lie algebras $\u^\geom_{g,n/\bar k}$  and $\u^\geom_{\cC_{g,n/\bar k}}$ by $\u^\geom_{g,n}$ and $\u^\geom_{\cC_{g,n}}$, respectively. 
The Lie algebra $\u^\geom_{g,n}$ is generated by $H_1(\u^\geom_{g,n})$,  which is determined by cohomology. More precisely, we have the following property of \textcolor{black}{the} relative completion.  
\begin{proposition}[{\cite[Thm.~3.8 ]{hain5}}]\label{relative comp iso on H_1}
	For each finite-dimensional $\Sp(H_\Ql)$-representation $V$, there is a natural isomorphism 
	$$\Hom_{\Sp(H)}(H_1(\u^\geom_{g,n}), V)\cong H^1(\pi_1(\M_{g,n/\overline{ \Q}}), V).$$
\end{proposition}
\subsubsection{Variant} Denote by $\cG^\geom_{\cC_{g,n}}$ the relative completion  of $\pi_1(\cC_{g,n/\bar k}, \bar x)$ with respect to the composition $\pi_1(\cC_{g,n/\bar k}, \bar x)\to \pi_1(\M_{g,n/\bar k}, \bar y)\to \Sp(H_\Ql)$. It is again an extension of $\Sp(H_\Ql)$ by a prounipotent $\Ql$-group denoted by $\U^\geom_{\cC_{g,n}}$. Denote the Lie algebras of $\cG^\geom_{\cC_{g,n}}$ and $\U^\geom_{\cC_{g,n}}$ by $\g^\geom_{\cC_{g,n}}$ and $\u^\geom_{\cC_{g,n}}$, respectively.  
%The analogue of Proposition \ref{comparison} also holds for $\g^\geom_{\cC_{g,n/\bar k}}$ and $\u^\geom_{\cC_{g,n/\bar k}}$ (see \cite[\S 8.1]{wat1}).

\subsection{Key exact sequences}Suppose that $g\geq 2$ and $n\geq 0$. 
%By Proposition \ref{comparison}, we may assume that  $k$ is a field of characteristic zero here. 
We consider the universal families $\pi:\cC_{g,n/k}\to \M_{g,n/k}$, $\pi^o:\M_{g,n+1/k}\to\M_{g,n/k}$, and $h:\M_{g,n/k}\to \M_{g/k}$. 
Recall that we fix a geometric point $\bar y$ of $\M_{g,n/k}$ and denote by $\Pi$ the fundamental group $\pi_1(C_{\bar y}, \bar x)$, where $C_{\bar y}$ is the fiber of $\pi$ over $\bar y$ and $\bar x$ is a geometric point of $C_{\bar y}$. Similarly, we denote $\pi_1(C^o_{\bar y}, \bar x)$ by $\Pi'$. We also consider $\bar y$ as a geometric point of $\M_{g/k}$ via the morphism $h$. The fiber of $h$ over $\bar y$ is given by
$$F(C_{\bar y})_{g,n}:=C_{\bar y}^n-\bigcup_{i\not=j}\left\{x_i=x_j\right\}.$$
Let $\bar \eta$ be the geometric generic point of $F(C_{\bar y})_{g,n}$.  The fundamental group $\pi_1(F(C_{\bar y})_{g,n}, \bar \eta)$ of the fiber is isomorphic to the profinite completion of $\pi^\top_{g,n}$, the topological fundamental group of the configuration space $F_{g,n}$. Denote the fundamental group $\pi_1(F(C_{\bar y})_{g,n}, \bar \eta)$ by $\pi_{g,n}$. The $\ell$-adic unipotent completion of $\pi_{g,n}$ is isomorphic to  $\cP_{g,n}\otimes \Ql$ and thus we denote its Lie algebra by $\p_{g,n}$ as well. Denote also the $\ell$-adic unipotent completions of $\Pi$ and $\Pi'$ by $\cP$ and $\cP'$ and their Lie algebras by $\p$ and $\p'$, respectively.
% (\textcolor{red}{Briefly expain how the following results follow from the resutls from Hain's paper and mine.})
The center-freeness of $\cP$, $\cP'$, and $\cP_{g,n}$ (see \cite{ntu}) together with the exactness criterion \cite[Prop.~6.11]{hain2} gives
\begin{proposition}[{\cite[Prop.~8.6 ]{hain2}}]With notations as above, the families $\pi$, $\pi^o$, and $h$ induce the exact sequences:
$$0\to\p\to\g^\geom_{\cC_{g,n}}\to\g^\geom_{g,n}\to 0;$$
$$0\to\p'\to\g^\geom_{g,n+1}\to\g^\geom_{g,n}\to 0;$$
$$0\to \p_{g,n}\to\g^\geom_{g,n}\to\g^\geom_g\to 0.$$
Consequently, these sequences remain exact after replacing $\g^\geom$ with $\u^\geom$ and after applying the functor $\Gr^W_\bullet$.
\end{proposition}
%\subsection{$\GSp(H)$ representations} Suppose that $g\geq 1$ and $H$ is a vector space of dimension $2g$ over $\Q$. For a field $F\supset \Q$, set $H_F=H\otimes F$. We equip $H_F$ with an alternating form $\theta: \Lambda^2H_F\to F$. Associated to the group $\GSp(H_F)$, there is a surjective homomorphism $\tau:\GSp(H_F)\to \Gm_{/F}$ such that for $\alpha\in \GSp(H_F)$, we have $\alpha^*\theta=\tau(\alpha)\theta$. The group $\Sp(H_F)$ is, then, defined to be the kernel of $\tau$: we have an exact sequence
%$$1\to\Sp(H_F)\to\GSp(H_F)\to\Gm_{/F}\to 1.$$
%Denote by $F(r)$ the one-dimensional $\GSp(H_F)$-representation on which $\GSp(H_F)$ acts via the $r$th power of $\tau$. For a $\GSp(H_F)$-representation $V$, set $V(r)=V\otimes F(r)$. Fixing a generator of $F(1)$, we consider $\theta$ as a $\GSp(H)$-equivariant homomorphism, that is, $\theta:\Lambda^2H_F\to F(1)$.  Furthermore, we view $\GSp(H)$ as a group scheme over $\Q$. Since the split torus of $\GSp(H)$ is defined over $\Q$, each finite-dimensional irreducible $\GSp(H)$-representation remains irreducible after being tensored with an algebraically closed field $L\supset \Q$ and  finite-dimensional $\GSp(H_F)$-representations are obtained by tensoring $\GSp(H)$-representations with $F$. 

\subsubsection{Weights} 
Recall that   $H_\Ql=H^1_\et(C_{\bar y} ,\Ql(1))$ \textcolor{black}{is} equipped with the cup product $\theta:\Lambda^2H_\Ql\to \Ql(1)$ and that we defined a central cocharacter $\omega:\Gm\to\GSp(H_\Ql)$ by setting $z\mapsto z^{-1}\id_H$. With respect to  $\omega$, each finite-dimensional irreducible $\GSp(H_\Ql)$-representation $V$ admits \textcolor{black}{a weight} $w(V)$. In particular, the representation $H_\Ql$ is of weight $-1$ and the representation $\Ql(r)$ is of weight $-2r$. \textcolor{black}{Then, the} finite-dimensional representation $V(r)$ has weight $w(V)-2r$. 

\subsubsection{Key representations}  Unless stated otherwise, for simplicity, we set $H=H_\Ql$. Here, we introduce $\GSp(H)$-representations that play an important role  in this paper. Firstly, the dual $\thetadual$ of $\theta$ can be viewed as a $\GSp(H_\Ql)$-equivariant map $\thetadual: \Ql(1)\to \Lambda^2H$ or equivalently a vector in $(\Lambda^2H)(-1)$. The key representations  appear in $\Gr^W_{m}\u^\geom_{g,n}$ for $m=-1,-2$. In weight $-1$, we consider the representations $H$ and $(\Lambda^3H)(-1)/(H\wedge \thetadual)$. Adopting Hain's notation from \cite{hain2}, we denote the representation $(\Lambda^3H)(-1)/H\wedge \thetadual$ by $\Lambda^3_0 H$. In weight $-2$, the main representations are $\Ql(1)$ and $\Lambda^2H/(\im \thetadual)$, denoted by $\Lambda^2_0H$. 

\subsection{The Lie algebras $\d_{g,n}$ and $\d_{\cC_{g,n}}$} \textcolor{black}{Before  introducing the description of the two-step nilpotent Lie algebras $\d_{g,n}$ and $\d_{\cC_{g,n}}$ in  \S\ref{two step lie}, we  review the $\GSp(H)$-modules $\Gr^W_m\u^\geom_{g,n}$ for $m=-1, -2$.} We will make a slight modification to the original definition of $\d_{g,n}$ made by Hain in \cite{hain2} and discuss a key result on the sections of the Lie algebra surjection $\d_{\cC_{g,n}}\to \d_{g,n}$ induced by $\pi:\cC_{g,n/k}\to\M_{g,n/k}$. 
\subsubsection{The $\GSp(H)$-modules $\p$, $\p'$, and $\p_{g,n}$} \textcolor{black}{Denote the pro-$\ell$ completion of $\Pi$ and $\Pi'$ by $\Pi^{(\ell)}$ and $\Pi'^{(\ell)}$, respectively.}  We note that the conjugation action of $\pi_1(\cC_{g,n/k}, \bar y)$ on $\Pi^\prol$ (resp. $\pi_1(\M_{g,n+1/k}, \bar y)$  on  $\Pi'^\prol$) induces an action of $\cG_{\cC_{g,n}}$ on $\p$ (resp. $\cG_{g,n+1}$ on $\p'$) and hence a weight filtration on $\p$ (resp. $\p'$). A key to understanding the $\GSp(H)$-module structure of $\Gr^W_\bullet\u^\geom_{g,n}$ is that of $\Gr^W_\bullet\p$ and $\Gr^W_\bullet\p'$.  \textcolor{black}{It follows directly from the main result of Labute \cite{lab} } that there is a minimal presentation 
$$\Gr^L_\bullet\p\cong \L(H)/(\im \thetadual),$$
where $L^\bullet\p$ is the lower central series of $\p$ with $L^1\p=\p$. 
\begin{proposition}
	With notations as above, the weight filtration  of $\p$ agrees with its lower central series, that is, for each $r\geq 1$,	$W_{-r}\p=L^r\p$. 
\end{proposition}
\begin{proof}
	Note that the map $\p^{\otimes r}\to L^r\p$ induced by the bracket is surjective. Thus we have $L^r\p\subset W_{-r}\p$ and furthermore, since $H=H_1(\p)$ is pure of weight -1,  $\Gr^L_r\p$ is pure of weight $-r$.  By strictness, the functor $V\mapsto W_mV$  is exact on the category of finite-dimensional $\cG_{g,n}$-modules and hence there is the exact sequence
	$$0\to W_{-r-1}L^{r+1}\p\to W_{-r-1}L^{r}\p\to W_{-r-1}\Gr^L_{r}\p\to 0.$$
	Since $\Gr^L_{r}\p$ is pure of weight $-r$, we see that $W_{-r-1}L^{r+1}\p = W_{-r-1}L^{r}\p$. We have $W_{-2}L^2\p=L^2\p=W_{-2}\p$ and then inductively, we have our claim. 	
\end{proof}
\textcolor{black}{\begin{remark} Note that a similar statement holds for any pronilpotent Lie algebra $\n$ with pure $H_1(\n)$. 
\end{remark}}
%An important fact here is that the weight filtration on $\p$ agrees with the lower central series of $\p$. This is because $H_1(\p)\cong H$ is pure of weight $-1$ and the bracket  is a morphism in the category of $\cG_{g,n}$-modules, which has the strictness property (see \cite[\S3.4]{wei}).  
Therefore, we have the following description of $\Gr^W_r\p$ for $r=-1,-2$:
$$\Gr^W_{-1}\p\cong H\hspace{.7in} \Gr^W_{-2}\p\cong \Lambda^2H/(\im \thetadual) = \Lambda^2_0H$$
The open immersion $j:C^o_{\bar y}\to C_{\bar y}$ induces a surjection  $\Pi'\to\Pi$, which, then, induces a  Lie algebra surjection $\p'\to \p$ that preserves weight filtrations. The Lie algebra surjection of the associated graded Lie algebras $\Gr^W_\bullet dj_\ast: \Gr^W_\bullet\p'\to \Gr^W_\bullet\p$ is an isomorphism in weight $-1$ and we have  $\Gr^W_{-2}\ker dj_\ast =\Ql(1)^{\oplus n}$:
$$\Gr^W_{-1}\p'\cong H \hspace{.7in}\Gr^W_{-2}\p'\cong \Lambda^2_0H\oplus\Ql(1)^{\oplus n}$$
Similarly, the conjugation action  of $\pi_1(\M_{g,n/k}, \bar y)$ on $\pi_{g,n}$ induces an action of $\cG_{g,n}$ on $\p_{g,n}$. Hence, the Lie algebra $\p_{g,n}$ admits a \textcolor{black}{natural} weight filtration \textcolor{black}{as a $\cG_{g,n}$-module}. By Proposition \ref{abelianization pure weight -1}, there is an isomorphism 
$$H_1(\p_{g,n})\cong \bigoplus_{j=1}^nH_j,$$ 
on which $\cG_{g,n}$ acts diagonally via $\GSp(H)$, and thus $H_1(\p_{g,n})$ is pure of weight $-1$. Therefore, the weight filtration again agrees with its lower central series, which allows us to use Theorem \ref{presentation for configuration space} to understand the bracket  of the associated graded Lie algebra $\Gr^W_\bullet\p_{g,n}$.
\subsubsection{The $\GSp(H)$-modules $\Gr^W_m\u^\geom_{g,n}$ for $m=-1,-2$} Here we assume that $g\geq 3$. The case $g=2$ will be treated in a forthcoming paper. \textcolor{black}{The functor $\Gr^W_\bullet$ on the category of $\cG_{g,n}$-modules is exact (see \cite[Thms. 3.9 and 3.12]{wei}). Hence,   we }can determine the $\GSp(H)$-module structure of each $\Gr^W_m\u^\geom_{g,n}$ by considering the exact sequences
$$0\to \Gr^W_\bullet\p'\to \Gr^W_\bullet\u^\geom_{g,n+1}\to \Gr^W_\bullet\u^\geom_{g,n}\to 0.$$ In this paper, we will only need to know $\Gr^W_m\u^\geom_{g,n}$ for $m=-1$ and $-2$ for our purposes. For $m=-1$, there are isomorphisms of $\GSp(H)$-modules
$$\Gr^W_{-1}\u^\geom_{g,n}\cong \Gr^W_{-1}H_1(\u^\geom_{g,n})= H_1(\u^\geom_{g,n})\cong \Lambda^3_n H:=\Lambda^3_0H\oplus\bigoplus_{j=1}^n H_j.$$
%\textcolor{red}{remove? : In particular, in the case when $n=1$, we have an isomorphism
%$$H_1(\u^\geom_{g,1})\cong (\Lambda^3 H)(-1)\cong \Lambda^3_0 H \oplus H\subset \Hom(H, \Lambda^2_0H),$$
%which is essentially due to Johnson's work on the abelianization  of the Torelli groups \cite{joh} along with Proposition \ref{relative comp iso on H_1}.}
Since $\Gr^W_\bullet\u^\geom_{g,n}$ is a graded Lie algebra with negative weights, there is a minimal presentation
$$\L(\Lambda^3_n H)/R\cong \Gr^W_\bullet\u^\geom_{g,n},$$
where $R$ is a Lie ideal generated in weight $-2$ and below. %for $g\geq 4$  and in weight $-2$ and $-3$ for $g=3$ (see \cite{hain3}). 
  Since $H_1(\u^\geom_{g,n})$ is pure of weight $-1$, the weight filtration on $\u^\geom_{g,n}$ agrees with its lower central series.  As to $\Gr^W_{-2}\u^\geom_{g,n}$, it comes down to understanding the bracket  $[~,~]:\Lambda^2\Gr^W_{-1}\u^\geom_{g, 1}=\Lambda^2\Lambda^3_1H\to \Gr^W_{-2}\u^\geom_{g,1}$, which was computed by Hain in \cite{hain0}. In summary, we have
\begin{proposition}[{\cite[Thm.~9.11]{hain2}}] \label{structure of u} Suppose that $g\geq 3$ and $n\geq 0$.  There are isomorphisms
\begin{enumerate}
\item $$\Gr^W_m\u^\geom_{g,n}\cong \begin{cases}
                                                       H_1(\u^\geom_{g,n})\cong\Lambda^3_n H=\Lambda^3_0H\oplus \bigoplus_{j=1}^nH_j& \text{ for } m=-1;\\
                                                        V_{[2+2]} (-1)\oplus \bigoplus_{j=1}^n\Lambda^2_0H_j\oplus \Ql(1)^{\oplus \binom{n}{2}} &\text{ for } m=-2;
                                                       
                                                      \end{cases}$$
\item $$\Gr^W_m\u^\geom_{\cC_{g,n}}\cong \begin{cases}
                                                       \Lambda^3_{n+1} H=\Lambda^3_0H\oplus \bigoplus_{j=0}^nH_j& \text{ for } m=-1;\\
                                                        V_{[2+2]} (-1)\oplus \bigoplus_{j=0}^{n}\Lambda^2_0H_j\oplus  \Ql(1)^{\oplus \binom{n}{2}} &\text{ for } m=-2,
                                                       
                                                      \end{cases}$$
where $V_{[2+2]}$ is the isomorphism class of $\Sp(H)$-representation corresponding to the partition $2+2$. 
\end{enumerate}
\end{proposition}

\subsubsection{The definitions of $\d_{g,n}$ and $\d_{\cC_{g,n}}$ and the sections of $\d_{\cC_{g,n}}\to \d_{g,n}$}{\label{two step lie}} We set 
$$\d_{g,n}:=\Gr^W_\bullet(\u^\geom_{g,n}/W_{-3}+V_{[2+2]}(-1)),$$
and 
$$\d_{\cC_{g,n}}:=\Gr^W_\bullet(\u^\geom_{\cC_{g,n}}/W_{-3}+V_{[2+2]}(-1)).$$
Note that our definition of $\d_{g,n}$ slightly differs from Hain's in \cite[\S 10]{hain2}. We make this definition due to a key role of the $\Sp(H)$ trivial representation $\Ql(1)$, dropping the representation $V_{[2+2]}(-1)$
% \textcolor{red}{remove?(\footnote{The representation $V_{[2+2]}(-1)$ will play an important role in the case of $g=2$ and the moduli of hyperelliptic curves. })}
 in this paper.  We have the following description of $\d_{g,n}$ and $\d_{\cC_{g,n}}$:
$$\Gr^W_{-1}\d_{g,n}\cong\Lambda^3_nH, \,\,\,\Gr^W_{-2}\d_{g,n}\cong \bigoplus_{j=1}^n\Lambda^2_0H_j\oplus  \Ql(1)^{\oplus \binom{n}{2}},\text{ and }\Gr^W_m\d_{g,n}=0 \text{ for }m\leq -3 ;$$
$$\Gr^W_{-1}\d_{\cC_{g,n}}\cong\Lambda^3_{n+1}H, \,\,\,\Gr^W_{-2}\d_{\cC_{g,n}}\cong \bigoplus_{j=0}^{n}\Lambda^2_0H_j\oplus  \Ql(1)^{\oplus \binom{n}{2}},\text{ and }\Gr^W_m\d_{\cC_{g,n}}=0 \text{ for }m\leq -3 .$$
%We observe that the open immersion $\M_{g,n+1/k}\to\cC_{g,n/k}$ induces a $\GSp(H)$-equivariant graded Lie algebra surjection $\Gr^W_\bullet\u^\geom_{g,n+1}/W_{-3}\to\Gr^W_\bullet\u^\geom_{\cC_{g,n}}/W_{-3}$, whose kernel is isomorphic to $\bigoplus_{j=1}^n\Ql(1)=\Gr^W_{-2}\ker(\p'\to\p)$. It then induces a $\GSp(H)$-equivariant graded Lie algebra surjection $\d_{g,n+1}\to\d_{\cC_{g,n}}$ such that 
We observe that the open immersion $\M_{g,n+1/k}\to\cC_{g,n/k}$ induces a $\GSp(H)$-equivariant graded Lie algebra surjection
$\d_{g,n+1}\to\d_{\cC_{g,n}}$ that makes the diagram 
$$\xymatrix@C=1em@R=1em{
\d_{g,n+1}\ar[r]\ar[d]^{\gamma^o_n}&\d_{\cC_{g,n}}\ar[d]^{\gamma_n}\\
\d_{g,n}\ar@{=}[r]                       &\d_{g,n}
}
$$
commutes,
%$$\xymatrix{
%\M_{g,n+1}\ar[r]\ar[d]_{\pi^o}&\cC_{g,n}\ar[d]^{\pi}\\
%\M_{g,n}\ar@{=}[r]&\M_{g,n}
%}
%$$
%induces the commutative diagram:
%$$\xymatrix{
%\d_{g,n+1}\ar[r]^\cong\ar[d]&\d_{\cC_{g,n}}\ar[d]\\
%\d_{g,n}\ar@{=}[r]                       &\d_{g,n}
%,}
%$$
where the left and right vertical maps are $\GSp(H)$-equivariant Lie algebra surjections induced by $\pi^o$ and $\pi$, respectively. Denote the projection $\d_{g,n+1}\to\d_{g,n}$ and $\d_{\cC_{g,n}}\to\d_{g,n}$ by $\gamma_n^o$ and $\gamma_n$, respectively.
%By universal property, each section of $\pi$ induces a section of $\cG_{\cC_{g,n}}\to\cG_{g,n}$, which yields a section of $\gamma_n$ by construction. 
The $\GSp(H)$-module structures of $\d_{\cC_{g,n}}$ and $\d_{g,n}$ together with the bracket computation $[~,~]:\Lambda^2\Gr^W_{-1}\d_{g,n}\to \Gr^W_{-2}\d_{g,n}$ (\cite[Prop.~9.13]{hain2}) give us 
\begin{proposition}[{\cite[Prop.~10.8]{hain2}}]\label{sections of dgn}
With notations as above, if $g\geq 4$, there are exactly $n$ $\GSp(H)$-equivariant Lie algebra sections of $\gamma_n$ given on $\Gr^W_{-1}$ by
$$s_j: \Lambda^3_0H\oplus H_1\oplus\cdots\oplus H_n\to \Lambda^3_0H\oplus H_0\oplus H_1\oplus\cdots\oplus H_n$$
$$(v; u_1,\ldots,u_n)\mapsto (v; u_j,u_1,\ldots,u_n),$$
for $j=1,\ldots,n$, where $v\in \Lambda^3_0H$ and $u_i\in H_i$ for each $i$. In particular, there is no section of $\gamma_0$. 
\end{proposition}
\begin{remark}A key fact used in the proof of this result is that for $g\geq 4$, the bracket $[~,~]:\Lambda^2\Lambda^3_0H\to \Lambda^2_0H$ is nontrivial. 
For the case $g=3$, the representation $\Lambda^2\Lambda^3_0H$ does not contain a copy of $\Lambda^2_0H$ and for $g=2$, the representation $\Lambda^3_0H$ is trivial. For this reason, we will need to consider weights down to $-4$ and the bracket $[~,~]:\Lambda^2V_{[2+2]}(-1)\to \Gr^W_{-4}\p$, which is discussed in \cite{wat2}.
\end{remark}
The following result is an analogue of Proposition \ref{sections of dgn} for the projection $\d_{g,n+1}\to\d_{g,n}$ and essential to the proof of Theorem 1.
\begin{proposition}\label{no sections of dgn lie algebra projections}
With notations as above, if $g\geq 4$, the projection $\gamma^o_n:\d_{g,n+1}\to\d_{g,n}$ does not admit any $\GSp(H)$-equivariant Lie algebra section.
\end{proposition}
\begin{proof} \textcolor{black}{Assume that $g\geq 4$.} For $n=0$, there is an isomorphism $\d_{g,1}\cong\d_{\cC_g}$ and hence our claim follows from Proposition \ref{sections of dgn}.
Suppose that there is a $\GSp(H)$-equivariant Lie algebra section $s$ of $\gamma^o_n$. By composing with the surjection $\d_{g,n+1}\to\d_{\cC_{g,n}}$, we obtain a section of $\gamma_n$.  \textcolor{black}{Then, by Proposition \ref{sections of dgn},} the restriction of $s$ to weight $-1$ is given by
$$\Gr^W_{-1}s: (v; u_1,\ldots,u_n)\mapsto (v; u_i,u_1,\ldots,u_n)$$
for some $i\in\{1,\ldots,n\}$.  \textcolor{black}{Recall notations in Theorem \ref{presentation for configuration space}.} For simplicity, we may assume $i=1$. First, when $n=1$, we see from the description of $\d_{g,1}$ that $\Gr^W_{-2}\d_{g,1}$ does not contain $\Ql(1)$ and hence that $\Theta_1=0$ in $\Gr^W_{-2}\d_{g,1}$. \textcolor{black}{Therefore, we have $0=\Gr^W_{-2}s(\Theta_1)= \sum_{k=1}^g[\Gr^W_{-1}s(a^{(1)}_k), \Gr^W_{-1}s(b^{(1)}_k)]$, since $s$ is a Lie algebra homomorphism.  On the other hand,} using Theorem \ref{presentation for configuration space}, 
%\textcolor{red}{remove?: and the notations used in it}
we have
\begin{align*}
\sum_{k=1}^g[\Gr^W_{-1}s(a^{(1)}_k), \Gr^W_{-1}s(b^{(1)}_k)]&=\sum_{k=1}^g[a^{(0)}_k+a^{(1)}_k, b^{(0)}_k+b^{(1)}_k]\\
                                                                                      &=\sum_{k=1}^g[a^{(0)}_k, b^{(0)}_k]+[a^{(0)}_k, b^{(1)}_k]+[a^{(1)}_k,b^{(0)}_k]+[a^{(1)}_k,b^{(1)}_k]\\
                                                                                      &=\Theta_0+ 2\Theta_{01} +\Theta_1\\
                                                                                      &=-\frac{1}{g}\Theta_{01}+2\Theta_{01}-\frac{1}{g}\Theta_{01}\\
                                                                                      &=\frac{2g-2}{g}\Theta_{01}. 
\end{align*}
This is nontrivial in $\Gr^W_{-2}\d_{g,2}$.
% \textcolor{red}{remove? (when $g\geq 4$)}. \textcolor{red}{remove? (Since $s$ is a Lie algebra homomorphism,)}
   \textcolor{black}{Thus, }we have a desired contradiction. Now assume $n>1$. Let $j\in\{2,\ldots,n\}$. First, we compute $\Gr^W_{-2}s(\Theta_{1j})$. We have
\begin{align*}
\Gr^W_{-2}s(\Theta_{1j}) &=\Gr^W_{-2}s\left(\sum_{k=1}^g[a^{(1)}_k,b_k^{(j)}]\right)\\
                                         &=\sum_{k=1}^g[\Gr^W_{-1}s(a^{(1)}_k),\Gr^W_{-1}s(b^{(j)}_k)]\\
                                         &=\sum_{k=1}^g[a^{(0)}_k+a^{(1)}_k, b^{(j)}_k]\\
                                         &=\sum_{k=1}^g[a^{(0)}_k,b^{(j)}_k]+[a^{(1)}_k,b^{(j)}_k]\\
                                         &=\Theta_{0j}+\Theta_{1j}. 
\end{align*}
For each $0\leq i<j\leq n$, fix a $\GSp(H)$-equivariant projection $q_{ij}:\Gr^W_{-2}\d_{g,n+1}\to \Ql(1)_{ij}$. Since by Theorem \ref{presentation for configuration space} we have $\Theta_1=-\frac{1}{g}\sum_{j=2}^n\Theta_{1j}$ in $\Gr^W_{-2}\d_{g,n}$, the above computation implies that $(q_{01}\circ\Gr^W_{-2}s)(\Theta_1)=0.$  On the other hand, in $\Gr^W_{-2}\d_{g,n+1}$ we have
$$(q_{01}\circ\Gr^W_{-2}s)(\Theta_1)=q_{01}(\Theta_0+2\Theta_{01}+\Theta_1)=\frac{2g-2}{g}\Theta_{01},$$
which is nontrivial. Thus we have a contradiction. 

\end{proof}

%%%%%%%%%%%%%%%%%%%%%%%%%%%%%%%%%%%%%%%%%%%%%%%%%%%%%%%%%%%%%%%%%%%%%%%%%%%%%%%%%%%%%%%%%%%%%%%%%%%%%%%%%%%%%%%%%%%%%%%
%%%%%%%%%%%%%%%%%%%%%%%%%%%%%%%%%%%%%%%%%%%%%%%%%%%%%%%%%%%%%%%%%%%%%%%%%%%%%%%%%%%%%%%%%%%%%%%%%%%%%%%%%%%%%%%%%%%%%%%
\section{Sections of  Universal Curves $\M_{g,n+1}\to\M_{g,n}$}
In this section, we will prove Theorem 1. Recall that  $\pi^o:\M_{g,n+1/k}\to\M_{g,n/k}$ is the complement in $\cC_{g,n/k}$ of the $n$ tautological sections of the universal curve $\pi:\cC_{g,n/k}\to\M_{g,n/k}$. When $\mathrm{char}(k)=0$, it follows from a result of Earle and Kra \cite{EaKr} that the sections of $\pi$ are exactly the $n$ tautological sections. This then implies that the $n$-punctured universal curve $\pi^o$ admits no sections. Theorem 1 is the analogue of this geometric result in terms of their fundamental groups. Although we are not able to determine the set of the conjugacy classes of the sections of $\pi_\ast: \pi_1(\cC_{g,n/k}, \bar x)\to \pi_1(\M_{g,n/k}, \bar y)$, the theorem gives us further evidence for \textcolor{black}{Grothendieck's section conjecture for the universal curve, that is, the bijection between  the set of the geometric sections  of $\pi$ and the set of the conjugacy classes of the group theoretical sections of $\pi_\ast$.} 

\begin{proof}[Proof of Theorem 1]  \textcolor{black}{Fixing an embedding $\overline{\Q}\hookrightarrow \bar k$ determines an isomorphism
$$ \pi_1(\M_{g,n/\overline{\Q}}, \bar y)\cong \pi_1(\M_{g,n/\bar k}, \bar y),$$ 
which in turn induces an isomorphism on the relative completions. Thus, we may assume that $k =\Q$. }
% \textcolor{red}{remove?(Fix a prime number $\ell$ such that $\ell$-adic cyclotomic character $\chi_\ell:G_k\to \Zl^\times$ has infinite image)}.
%  We set $H=H_\Ql$. 
Let $s$ be a section of $\pi^o_\ast:\pi_1(\M_{g,n+1/\bar k}, \bar x)\to\pi_1(\M_{g,n/\bar k}, \bar y)$. By the universal property of relative completion (see \cite[cf.~\S 6.1, 6.2]{hain2}), $s$ induces a section $\tilde{s}_{\cG}$ of $\tilde{\pi}^o_\ast:\cG^\geom_{g,n+1}\to \cG^\geom_{g,n}$, which then induces a section $\tilde{s}_{\U}$ of $\U^\geom_{g,n+1}\to\U^\geom_{g,n}$ that makes the diagram
$$\xymatrix{
1\ar[r]&\U^\geom_{g,n+1}\ar[r]\ar[d]&\cG^\geom_{g,n+1}\ar[r]\ar[d]^{\tilde{\pi}^o_\ast}&\Sp(H)\ar[r]\ar@{=}[d]&1\\
1\ar[r]&\U^\geom_{g,n}\ar[r]\ar@/^1pc/[u]^{\tilde{s}_{\U}}&    \cG^\geom_{g,n}    \ar[r]\ar@/^1pc/[u]^{\tilde{s}_{\cG}}&\Sp(H)\ar[r]&1
}
$$
\textcolor{black}{commutative}. From this, we see that the section $\tilde{s}_{\U}$ induces a Lie algebra section $ds_\ast:\u^\geom_{g,n}\to\u^\geom_{g,n+1}$. Since the weight filtration on $\u^\geom_{g,n}$ agrees with its lower central series, it follows that $ds_\ast$ yields a graded Lie algebra section $\Gr^W_\bullet ds_\ast$ of $\Gr^W_\bullet\u^\geom_{g,n+1}\to \Gr^W_{\bullet}\u^\geom_{g,n}$, which then induces a graded Lie algebra section $\d(ds_\ast)$ of  $\gamma^o_n$.  
We will show that $\Gr^W_\bullet ds_\ast$ is an $\Sp(H)$-equivariant homomorphism. 
\textcolor{black}{Since $\Gr^W_\bullet \u^\geom_{g,n}$ is generated by $H_1(\u^\geom_{g,n})\cong H_1(\U^\geom_{g,n})$, it is enough to show that the induced map $\tilde{s}_{\U}^\ab$ on the abelianization  is $\Sp(H)$-equivariant. }
\textcolor{black}{Note that the $\Sp(H)$-action on $H_1(\U^\geom_{g,n})$ is induced by the restriction of the inner action of $\cG^\geom_{g,n}$ to $\U^\geom_{g,n}$. 
}
%For $m \leq -1$, let $v$ be a vector in $\Gr^W_m\u^\geom_{g,n}$ and $\tilde{v}$ be a representative of $v$ in $\U^\geom_{g,n}$. 
%Let $r$ be an element in $\Sp(H)$. Since the conjugation action of $\U^\geom_{g,n}$ on each graded quotient $\Gr^W_m\u^\geom_{g,n}$  is the identity,   we may represent $r$ by an element $\tilde{r}$ in $\cG^\geom_{g,n}$.  Let  $\log: \cG^\geom_{g,n}\to\g^\geom_{g,n}$ be the logarithm map from $\cG^\geom_{g,n}$ to its Lie algebra $\g^\geom_{g,n}$.  
\textcolor{black}{Let $r$ be an element in $\Sp(H)$. Represent $r$ by an element $\tilde{r}$ in $\cG^\geom_{g,n}$. For $v$ in $H_1(\U^\geom_{g,n})$, let $\tilde{v}$ be a representative of $v$ in $\U^\geom_{g,n}$.  Then we have
\begin{align*}
\tilde{s}^\ab_{\U}(r\cdot v) & = (\tilde{s}_{\U}(\tilde{r}^{-1}\tilde{v}\tilde{r}))^\ab\\
                                         & = (\tilde{s}_{\cG}(\tilde{r}^{-1})\tilde{s}_{\U}(\tilde{v})\tilde{s}_{\cG}(\tilde{r}))^\ab\\
                                         & = (\tilde{s}_{\cG}(\tilde{r})^{-1}\tilde{s}_{\U}(\tilde{v})\tilde{s}_{\cG}(\tilde{r}))^\ab\\
                                         & = r\cdot\tilde{s}^\ab_{\U}(v),
                                         \end{align*}
where $\phantom{}^\ab$ denotes the image in the abelianization $H_1(\U^\geom_{g,n+1})$.                                         }
Thus \textcolor{black}{the map $\tilde{s}^\ab_{\U}$ is $\Sp(H)$-equivariant.} Hence $\Gr^W_\bullet ds_\ast$ is an $\Sp(H)$-equivariant graded Lie algebra section and so is $\d(ds_\ast)$, but  by Proposition \ref{no sections of dgn lie algebra projections}, there is no such section of $\gamma^o_n$.  \textcolor{black}{This implies that there is no section of 
$\tilde{\pi}^o_\ast$, and therefore no section of $\pi^o_\ast$.}  \\
%\indent Secondly, assume that $k$ is a finite field of characteristic $p>0$. Let $\ell$ be a prime number distinct from $p$. We consider the exact sequence
%$$1\to \Pi'^\prol\to \pi_1(\M_{g,n+1/\bar k}, \bar x)'\to\pi_1(\M_{g,n/\bar k},\bar y)\to1.$$ 
%A similar argument as above shows that a section $s$ of $\pi_1(\M_{g,n+1/\bar k}, \bar x)'\to\pi_1(\M_{g,n/\bar k},\bar y)$ induces an $\Sp(H)$-equivariant Lie algebra section of $\gamma_n$, and hence our claim follows. 

\end{proof}
An immediate corollary of Theorem 1 is the following result:
\begin{corollary}
If $g\geq 4$ and $n\geq 0$, the exact sequence
$$1\to \Pi'^\top\to\G_{g,n+1}\to\G_{g,n}\to 1$$
%$$1\to \pi_1^\top(C_{\bar y}^o, \bar x)\to\pi_1^\orb(\M_{g,n+1/\C}, \bar x)\to \pi_1^\orb(\M_{g,n/\C}, \bar y)\to 1$$
does not split, where $\Pi'^\top$ denotes the topological fundamental group $\pi_1^\top(C^o_{\bar y}, \bar x)$. 
\end{corollary}
\begin{proof}A section of the projection $\pi_1^\orb(\M_{g,n+1/\C}, \bar x)\to \pi_1^\orb(\M_{g,n/\C}, \bar y)$ induces a section of $\pi_1(\M_{g,n+1/\C}, \bar x)\to \pi_1(\M_{g,n/\C}, \bar y)$ upon profinite completion. By Theorem 1, there is no such section, and hence our result follows. 
\end{proof}
\begin{remark}
For the case $n=0$, this is the Birman exact sequence 
$$1\to \Pi^\top\to\G_{g,1}\to\G_g\to 1,$$ which was known to be nonsplit (see  \cite[Cor.~5.11]{FaMa}). 
%(\textcolor{red}{cite from  a primer on mapping class groups by Farb and Margalit})
\end{remark}

%%%%%%%%%%%%%%%%%%%%%%%%%%%%%%%%%%%%%%%%%%%%%%%%%%%%%%%%%%%%%%%%%%%%%%%%%%%%%%%%%%%%%%%%%%%%%%%%%%%%%%%%%%%%%%%%%%%%%%%
%%%%%%%%%%%%%%%%%%%%%%%%%%%%%%%%%%%%%%%%%%%%%%%%%%%%%%%%%%%%%%%%%%%%%%%%%%%%%%%%%%%%%%%%%%%%%%%%%%%%%%%%%%%%%%%%%%%%%%%
%\section{Topological Section Conjecture for Universal Curves}

\end{document}